\documentclass[12pt]{amsart}
\allowdisplaybreaks
\usepackage[left=2.85cm,right=2.85cm,top=2.54cm,bottom=2.54cm]{geometry}
\usepackage{amsmath, amssymb, amsthm, amsbsy, mathrsfs}
\usepackage{mathtools,stackrel}
\usepackage[shortlabels]{enumitem}
\newlist{myenum}{enumerate}{1}
\setlist[myenum,1]{label*=($h_{\arabic*}$)}  
\setlist{leftmargin=13mm}
\usepackage[colorlinks=true,linkcolor=blue,
            citecolor=blue,urlcolor=blue]{hyperref}
\usepackage{xcolor}
\usepackage{microtype}
\newtheorem{lemma}{Lemma}

\newtheorem{theorem}{Theorem}
\newtheorem{corollary}{Corollary}
\theoremstyle{definition}
\newtheorem{definition}{Definition}

\newtheorem{example}{Example}
\DeclareMathAlphabet{\rscal}{U}{rsfs}{m}{n} 
\newcommand{\G}{\rscal{G}}
\newcommand{\E}{\mathbb{E}}
\newcommand{\Prb}{\mathbb P}
\newcommand{\R}{\mathbb R}

\newcommand{\Eq}{\boldsymbol{E}}
\newcommand{\myx}{{x}}
\newcommand{\myy}{{y}}

\newcommand{\D}{\Delta^{\! c}}

\newcommand{\edges}{{\vert E\vert}}
\newcommand{\Fix}{\textup{Fix}}
\newcommand{\CR}{\mathcal R}
\newcommand{\F}{\mathcal F}

\title[]{Multiple colour interacting urns on complete graphs}
\date{\today}

\begin{document} 
\author[B. Pires]{Benito Pires}
\author[R. A. Rosales]{Rafael A. Rosales}
\address[B. Pires, R. A. Rosales]{%
  Departamento de Computa\c{c}\~ao e Matem\'atica,
  Faculdade de Filosofia, Ci\^encias e Letras,
  Universidade de S\~ao Paulo,
  14040-901, Ribeir\~ao Preto - SP, Brazil
}
\email{benito@usp.br}
\email{rrosales@usp.br}
\thanks{The first author was partially supported by grant
  {\#}2019/10269-3 and  {\#}2022/14130-2
  S\~ao Paulo Research Foundation (FAPESP)
} 

\begin{abstract}
We present a multiple colour generalisation of the model of graph
interacting urns studied by Bena\"im \emph{et. al.}, Random
Struct. Alg., 46: 614-634, 2015. We show that for complete graphs and
for a broad class of reinforcement functions governing the
 {addition} of balls in the urns, the process of colour proportions
at each urn converges almost surely to the fixed points of the
 {reinforcement} function.
\end{abstract}

\keywords{generalized P\'olya urns, stochastic approximations,
  reinforced processes}

\subjclass[2020]{Primary: 60K35, 37C10; Secondary: 62L20, 37A50}

\maketitle


\section{Introduction}
Interacting P\'olya urn models have recently gained considerable
attention, see for instance \cite{AG17}, \cite{MCJJ22}, \cite{KS22},
\cite{LP19}, \cite{LL12} and \cite{RT22}. This article considers a
multiple colour gene\-ralisation of the model in \cite{BBCL15} when
the interaction is determined by the structure of a complete graph.
The model in \cite{BBCL15} is described as follows. Let $\G=(V,E)$ be
a finite, connected, undirected graph such that
$V=[d]=\{1,\ldots,d\}$. Each vertex $u\in V$ presents an urn with
$B_u(0) \geq 1$ balls at time $n=0$.  At each time $n \ge 1$ and to
each edge $\{u,v\}\in E$, one ball is added either to $u$ or $v$. The
ball is added to the urn at vertex $u$ with probability proportional
to $B_u(n)^\alpha$, where $B_u(n)$ denotes the number of balls in the
urn $u$ at time $n\geq 0$ and $\alpha > 0$ is a parameter that
reinforces the effect of the number of balls. \cite{BBCL15} is
concerned with strong laws for the proportion of balls in each urn
$X_v(n)$ depending upon $\G$ and $\alpha$. A complete characterisation
for finite, connected graphs in the linear case, i.e. when $\alpha =
1$, is given in \cite{MR3269167} and \cite{L16}.

To introduce our model, assume that an urn with $B^i_u(0)\ge 1$ balls
of colour $i \in C = \{1, 2, \ldots, c\}$, $c\ge 2$, is placed at each
vertex $u\in V$ of a complete graph $\G=(V,E)$. To simplify matters,
we assume that there exists $b_0 \ge 1$ such that $\sum_{u\in V}
B_u^i(0)=b_0$ for all $i\in C$, that is, we start the process with the
same total number of balls of each colour. The number of balls of
colour $i$ in the urn $u\in V$ at time $n\ge 0$ is denoted by
$B_u^i(n)$. Let $\F_n$, $n\ge 1$, be the $\sigma$-algebra
$\sigma(B_v^i(k); 1 \leq k \leq n, i \in C, v \in V)$. Denote by
\begin{equation}\label{process}
  X_u^i(n) = \frac{B_u^i(n)}{\sum_{v\in V} B_v^i(n)}
\end{equation}
 the proportion of balls of colour $i$ at time $n$ in the urn $u$. At
each step $n$ one ball of each colour is placed at each edge of
$\G$. For each $\{u, v\} \in E$, let $B_{v\to u}^i(n+1)$ be the event
defined as
\[
  B_{v\to u}^i(n+1) =
  \big\{\text{the ball of colour $i$ at edge $\{u,v\}$ goes to vertex
    $u$ at \text{time} $n+1$}\big\}.
\]
Conditionally on $\F_n$, the ball of colour $i$ placed at edge $\{u,
v\}$ at step $n+1$ is put into the urn at vertex $u$ according to the
probability
\begin{equation}\label{rule3}
  \Prb\big(B_{v\to u}^i(n+1) \mid \F_n\big)
  =
  \varphi\Big(\textstyle\sum_{k\in C{\setminus}\{i\}} \big(\alpha
  X_u^k(n) - \alpha X_v^k(n)\big)\Big)
\end{equation}
where $\varphi:\R \to (0,1)$ is any smooth function such that
\begin{myenum}[topsep=0.1cm, itemsep=0cm] 
\item there exists $M>0$ such that $0<\varphi'(t)<M$
  for all $t\in\mathbb{R}$;\label{deriv_Hyp} 
\item $\varphi(t) + \varphi(-t) = 1$ for all
  $t\in\mathbb{R}$.\label{odd_Hyp}
\end{myenum}

To gain some intuition about why this model may be interesting,
consider the case with two colours, $C=\{1,2\}$, with $\varphi$
defined by $t \mapsto 1/(1+e^{-t})$. In this case, for $i=1$, the
right hand side of \eqref{rule3} equals
\[
 \frac{\exp(\alpha
     X_u^2(n))}{\exp(\alpha X_u^2(n)) +   
   \exp(\alpha X_v^2(n))},
\]
and so, when $\alpha < 0$, the probability that a ball of colour $i=1$
is put into the urn $u$ decreases exponentially with the proportion of
balls of colour $2$. This model can be used to describe a
``competition" among the different colours for the urns. The case
$\alpha > 0$, on the other hand, defines a model for cooperative
dynamics. The model may be used to describe the dynamics of
residential segregation of individuals from $C$ groups into $d$
regions identified with the vertices of $\G$.  It may also be
conceived as a repeated population game among agents of $C$ types or
the competition of $d$ companies for $C$ products, providing a
generalisation of the examples of the single colour model in
\cite{BBCL15} when $\G$ is complete. A concrete example that served as
motivation is as follows. Suppose that three companies located at the
vertices of a triangle hire employees coming from nearby towns. In
this case, $\G = (V, E)$ is the complete graph with $V=\{1,2,3\}$
denoting the locations/labels of the companies and $E=\{\{1,2\},
\{1,3\}, \{2,3\}\}$ referring to the locations/labels of the
towns. Let $C=\{i,\bar{\imath}\}$ be a set of labels identifying
native Spanish and Portuguese speakers. The number of employees of
type $j\in C$ at company $v$ at time $n$ is $B_v^j(n)$ and $X_v^j(n)$
denotes the corresponding proportion relative to the total number of
employees of class $j$ in all three companies. At each month $n\ge 0$
and at each edge $\{u,v\}\in E$, a Spanish and a Portuguese speaking
candidate look for a job in one of the nearby companies $u$ and
$v$. Both companies have the same hiring policy that seeks, in the
long run, to equilibrate the number of employees speaking Spanish or
Portuguese. To be more precise, assume $\alpha>0$ and that the
candidate of type $i$ at edge $\{u,v\}$ at time $n$ will be paid US
$\$ 100{,}000\,\varphi\big(\alpha X_u^{\bar{\imath}}(n)-\alpha
X_v^{\bar{\imath}}(n)\big)$ as annual salary if he chooses company
$u$, and US $\$ 100{,}000\,\varphi\big(\alpha X_v^{\bar{\imath}}(n) -
\alpha X_u^{\bar{\imath}}(n)\big)$ if he chooses company $v$. Then, if
$X_u^{\bar{\imath}}(n)=X_v^{\bar{\imath}}(n)$, since
$\varphi(0)=\frac12$, both companies $u$ and $v$ will pay the same
salary. On the other hand, if $X_u^{\bar{\imath}}(n) >
X_v^{\bar{\imath}}(n)$ and if the candidate of type $i$ chooses the
company with the least number of employees of type $i$ (which means to
choose $u$), then the candidate will have a bigger salary, which
approaches the upper bound US~$\$ 100{,}000$ as $X_u^{\bar{\imath}}(n)
- X_v^{\bar{\imath}}(n)$ increases.

Strong laws for the proportions $X_v(n)$ are obtained in \cite{BBCL15}
via the stochastic approximation approach which essentially leads to
the study of a flow induced by an underlying vector field. The field
in \cite{BBCL15} is gradient. In our case, the underlying field is not
gradient, thus the induced dynamics is harder to be understood.  The
principal contribution of this article is the construction of a
Lyapunov function which allows us to establish global properties about
the vector field and thus almost sure convergence properties of
$X_v^i(n)$. The Lyapunov function is obtained by relating our field to
the one governing the evolution of Hopfield's neural networks, see
\cite{WOS:A1984SU47600030}. The main results of this article are
described in Section \ref{sec:Results}. Proofs are given in Sections
\ref{sec:edo}, \ref{sec:corol_proof} and \ref{sec:exampl_proof}.

\section{Statement of main results}\label{sec:Results}
By using \eqref{process}, define the \textit{process of colour
proportions} as
\begin{equation}\label{process2} 
  X= \{X(n)\}_{n\ge 0},\ \text{ where}\ \
  X(n)=\big(X^1_1(n),\ldots,  X_d^1(n)\ldots, X_1^c(n),\ldots, X_d^c(n)\big).
\end{equation} 
Observe that for each $n$, $X(n)$ takes values on
\[
  \Delta^c = \bigg\{(x_u^i)\in\mathbb{R}^{dc}: x_u^i\ge 0, \forall
    u,i,\,\, \textrm{and}\,\,  {\sum_{u=1}^d} x_u^i =1,\forall  i\bigg\}. 
\] 
We assume hereafter that the process $X$ is defined on an appropriate
probability space $(\Omega, \F, \Prb)$ endowed with the filtration
$\F_n$.

We are principally concerned with a strong law for $X$. To state our
main result, let $\pi = {(\pi_{u}^i):}\D\to\D$ be the map whose
coordinate-function at $\myx= \big(x_u^i\big)\in\D$ takes the value
\begin{equation}\label{pi}
\pi_u^i(\myx) = 
\dfrac{1}{\vert E\vert } \sum_{w\in V{\setminus}\{u\}}
\varphi\left(  \sum_{k\in C{\setminus}\{i\}}  \left(\alpha x_{u}^k -
    \alpha x_{w}^k\right)    \right).
\end{equation}
It follows from the hypotheses \ref{deriv_Hyp} and \ref{odd_Hyp} that
$\pi$ takes $\Delta^{\! c}$ into the interior of $\Delta^c$, as
$\pi_u^i>0$. Denote by Fix$(\pi)$ the set of fixed points of $\pi$,
i.e., Fix$(\pi) = \{x \in \mathbb{R}^{dc} : \pi(x) = x\}$.

\begin{theorem}\label{theorem1}
Let $\G = (V, E)$ be a finite complete graph, $X=\{X(n)\}_{n\ge 0}$ be
the process of colour proportions at each urn given by
\eqref{process2}, and $\pi$ be the map defined in \eqref{pi}. If the
set of fixed points of $\pi$ is finite, then $X=\{X(n)\}_{n\ge 0}$
converges almost surely to the fixed points of $\pi$, that is,
\[
  \sum_{x\in\  {\Fix}(\pi)} \mathbb{P}\left(\lim_{n\to\infty}
    X(n)=x\right)=1. 
\]       
\end{theorem}

As the next result clarifies, the hypothesis that $\pi$ has finitely
many fixed points in Theorem \ref{theorem1} is automatic when the
parameter $\alpha$ in \eqref{pi} has small absolute value. In this
case, there exists a unique asymptotic state, i.e., a unique point
$x\in\Delta^c$ such that $\lim_{n\to\infty} X(n)=x$. On the other
hand, it is an open question whether for all $\alpha$ satisfying
$\vert \alpha\vert\ge d/ \big(4 (c-1) \sup_{t\in\mathbb{R}}
\vert\varphi'(t)\vert \big)$, the number of fixed points of $\pi$ is
finite.

 \begin{corollary}\label{corollary1}
Let $\G = (V, E)$ be a finite complete graph, $X=\{X(n)\}_{n\ge 0}$ be
the process of colour proportions at each urn given by
\eqref{process2}, and $\pi$ be the map defined in \eqref{pi}. If
$\vert \alpha\vert< d\big/ \big(4 (c-1) \sup_{t\in\mathbb{R}}
\vert\varphi'(t)\vert \big)$, then $X=\{X(n)\}_{n\ge 0}$ converges
almost surely to the unique fixed point $x^*=\left(
    \frac1d,\ldots, \frac1d\right)\in\mathbb{R}^{dc}$ of $\pi$.
\end{corollary}

In what follows, we denote by $\rho(A) = \max\,\{\vert\lambda\vert:
\lambda\,\, \textrm{is eigenvalue of $A$}\}$ the spectral radius of a
square matrix $A$. Given $x\in\Delta^c$, we let $J\pi(x)$ denote the
Jacobian matrix of the map $\pi$ at $x$. We say that a fixed-point $x$
of $\pi$ is stable if $\rho\big(J\pi(x)\big)<1$. If $\alpha$ has large
absolute value, then $\pi$ may have multiple stable fixed points, as
the following example shows.

\begin{example}\label{example1}
Let $V=\{1,2,3\}$, $C=\{1,2\}$, $\alpha=\frac{615}{182}\ln 9$,
$\varphi(t)=1/(1+e^{-t})$, $t\in\mathbb{R}$, and $\pi$ be the map
defined in \eqref{pi}. Then $ \textrm{Fix}\,(\pi)$ contains the stable
fixed points:
\begin{equation}\label{eqn:example_fp}
 \begin{aligned}
 \textstyle\left(\frac{23}{615}, \frac{1}{3},\frac{129}{205},\frac{23}{615},
    \frac{1}{3},\frac{129}{205}  \right), &
&\textstyle\left( \frac{23}{615}, \frac{129}{205}, \frac{1}{3},
  \frac{23}{615},\frac{129}{205}, \frac{1}{3}  \right), &
&\textstyle\left(\frac13, \frac{23}{615},\frac{129}{205}, \frac13,
  \frac{23}{615},\frac{129}{205}  \right) \\[0.1in]
 \textstyle\left( \frac{129}{205}, \frac{1}{3},\frac{23}{615},\frac{129}{205},
  \frac{1}{3},\frac{23}{615}  \right), &
&\textstyle\left( \frac{129}{205}, \frac{23}{615},\frac13, \frac{129}{205},
  \frac{23}{615}, \frac13  \right), &
&\textstyle\left(\frac13,  \frac{129}{205}, \frac{23}{615},\frac{1}{3},
  \frac{129}{205},\frac{23}{615}\right)
 \end{aligned}
\end{equation}

Figure~\ref{fig:triangles} presents six sample paths for a single
colour of the process $X=\{X(n)\}_{n\ge 0}$ started at $X(0) =
\left(\frac13,\frac13,\frac13,\frac13,\frac13,\frac13 \right)$. The
figure includes simulations for two different reinforcement strengths,
$\alpha=\frac34$ and $\alpha = ({615}/{182})\ln 9$. Open circles are
located at the fixed points of $\pi$. In the later case, the six
non-central circles are located at the points with coordinates
determined by \eqref{eqn:example_fp}. Continuous light curves
represent the flow of the ODE in \eqref{eqn:THE_field}. These
simulations  are consistent with the general result stated
in Theorem~\ref{theorem1} and also with Corollary~\ref{corollary1}.
\begin{figure}[h!]
  \begin{tabular}{cc}
     \includegraphics[width=6cm]{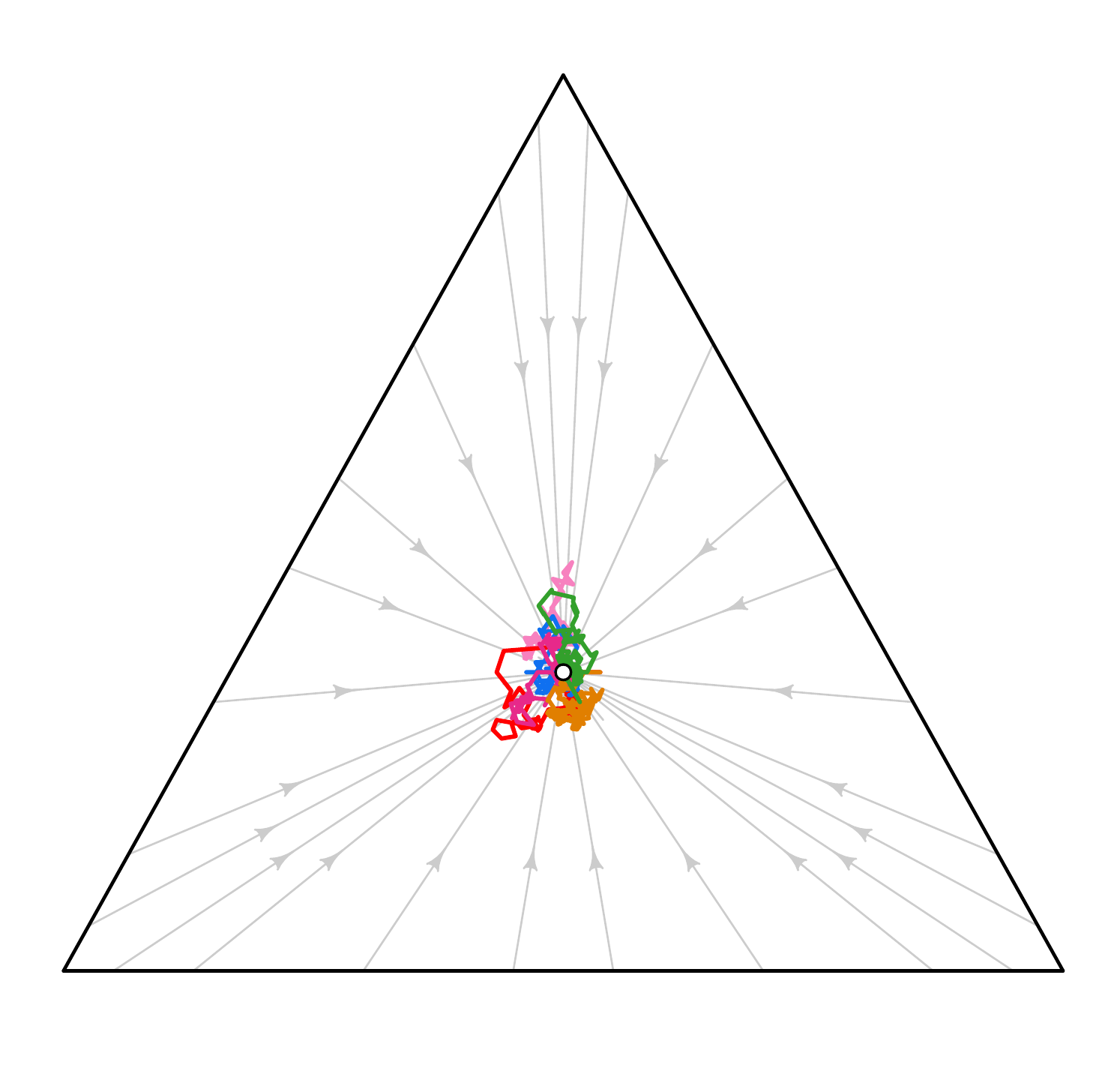}
    &
     \includegraphics[width=6cm]{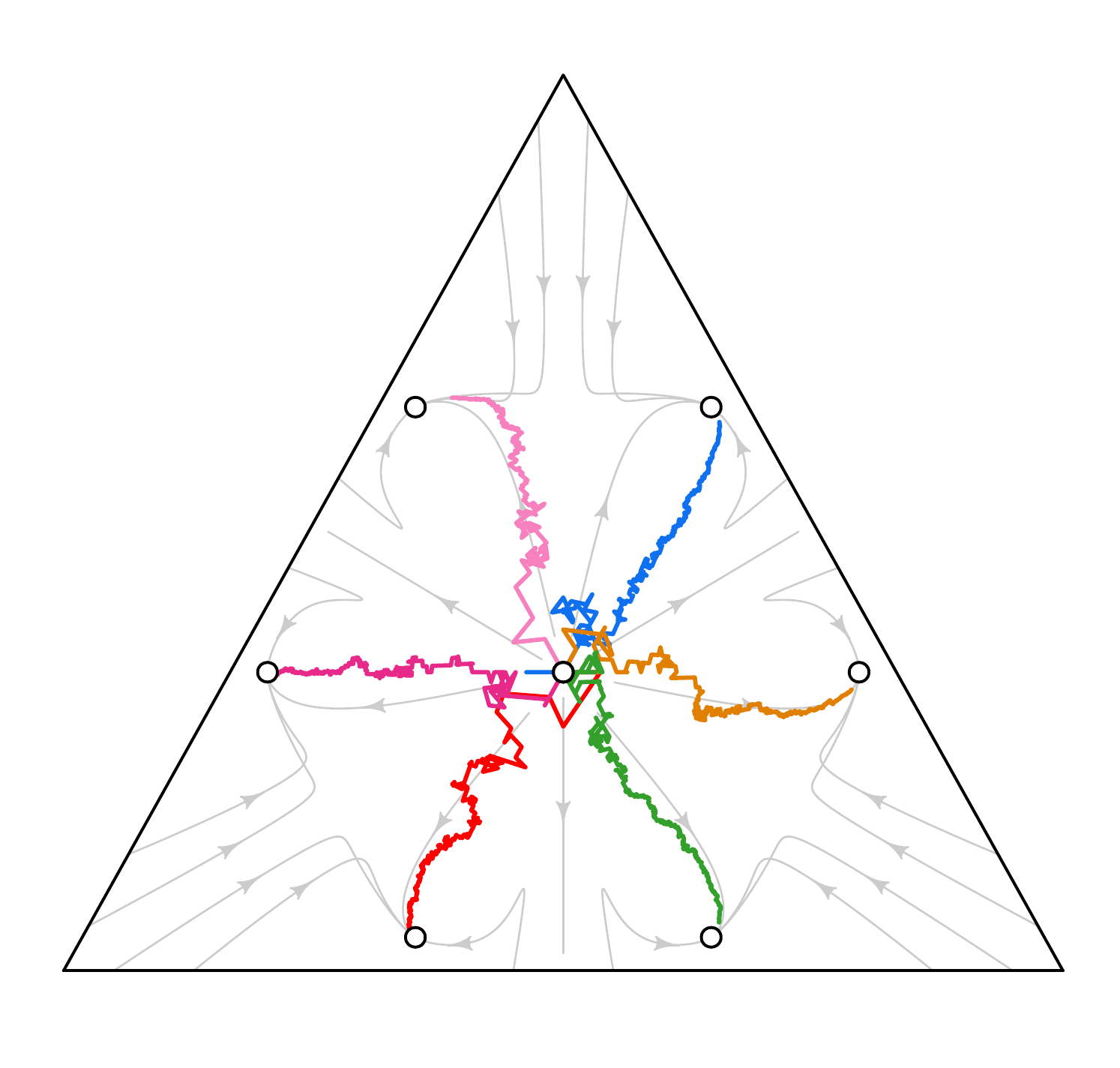}
    \\ 
    $\alpha = \frac34$ & $\alpha = \frac{615}{182}\ln 9$ 
  \end{tabular}
  \caption{Sample paths for a single colour and different
    reinforcement parameters $\alpha$.}\label{fig:triangles}
\end{figure}
\end{example}

\section{Proof of Theorem~\ref{theorem1}}\label{sec:edo}
The almost sure convergence of $X$ is established by using the ODE
method for stochastic approximations (see Appendix A). This technique
shows that the limit set of the process $X$ is almost surely connected
and contained in the chain-recurrent set (see Definition~\ref{dcs}) of
the ordinary differential equation
 \begin{equation}\label{eqn:THE_field}
  \dot{x} = F(x):= -x + \pi(x), \quad x\in\D.
\end{equation}
If $\pi$ has finitely many fixed points, then the chain-recurrent
set of \eqref{eqn:THE_field} equals the set of fixed points of $\pi$
(see Corollary \ref{cor1}). To achieve this result, we construct a
strict Lyapunov function for \eqref{eqn:THE_field} in terms of energy
functions of Hopfield's Neural Networks.

We will need some notions from the theory of ordinary differential
equations. The function $F$ in \eqref{eqn:THE_field} is defined on the
entire $\mathbb{R}^{dc}$. By \ref{deriv_Hyp} and \ref{odd_Hyp}, $F$ is
a smooth map with bounded partial derivatives. Under such conditions,
there exists a semi-flow associated to \eqref{eqn:THE_field} and
denoted by $\phi_t:\mathbb{R}^{dc}\to\mathbb{R}^{dc}$, $t\ge 0$. A set
$S\subset\mathbb{R}^{dc}$ is called \textit{invariant} by
\eqref{eqn:THE_field} if $\phi_t\big(\myx\big)\in S$ for all $t\ge 0$
and $\myx\in S$. A point $\myx\in\mathbb{R}^{dc}$ is an
\textit{equilibrium point} if $\phi_t(\myx)=\myx$ for all $t\ge0$. We
denote by $\Eq$ the set of equilibrium points of
\eqref{eqn:THE_field}.

\begin{definition}[strict Lyapunov function]\label{slf} A continuous
function $L:S\to\mathbb{R}$ is a \emph{strict Lyapunov function} for
\eqref{eqn:THE_field} on an invariant set $S\subset\mathbb{R}^{dc}$ if
the function \mbox{$t\in [0,\infty)\mapsto L\big(\phi_t(\myx)\big)$}
is strictly decreasing for all ${x}\in S{\setminus} \Eq$.
\end{definition} 

\begin{theorem}\label{theorem2}  {The following statements concerning
the differential equation \eqref{eqn:THE_field} hold:}
\begin{itemize}[noitemsep]
\item [$(i)$] $\D\subset \mathbb{R}^{dc}$ is an invariant set; \item
[$(\mathit{ii})$] There exists a strict Lyapunov function $L:\D\to
\mathbb{R}$ on $\D$.
\end{itemize}
\end{theorem}

\begin{proof}[Proof of Theorem~\ref{theorem2}.$(i)$] Let $\myx=(x_u^i)\in\Delta^{\!
c}$, that is, assume that $x_u^i\ge 0$ for all $u\in V$, $i\in C$, and
$\sum_{u\in V} x_u^i=1$ for all $i\in C$. By \eqref{pi} and
\ref{odd_Hyp}, since $\G=(V,E)$ is an undirected complete graph, we
have that
\begin{align*}
  \sum_{u\in V} \pi_u^i(\myx)
  =
    \dfrac{1}{\edges} \sum_{\substack{u, w\in V\\ u < w}}
    \Bigg[&\varphi\Bigg(\sum_{k\in C{\setminus}\{i\}}  \Big(\alpha
    x_{u}^k - \alpha x_{w}^k\Big)\Bigg) \\
  +\ \  &\varphi\Bigg(
  \sum_{k\in C{\setminus}\{i\}}  \Big(\alpha x_{w}^k -  \alpha 
    x_{u}^k\Big)\Bigg)\Bigg]                     
  =
   {\dfrac{1}{\edges}\, \edges=1}.
\end{align*} 
In this way, $\pi(\myx)$ takes $\Delta^{\! c}$ into its interior
because $\pi_u^i(\myx)>0$ for all $u,i $, and $\sum_{u\in V}
\pi_u^i(\myx)=1$ for all $i\in C$. In particular, if $\myx$ is a point
in the boundary of $\Delta^{\! c}$, we have that $F(\myx) = - \myx +
\pi(\myx)$ is the displacement vector from $\myx\in\partial{\Delta^{\!
c}}$ to the point $\pi\big({\myx}\big)$ in the interior of $\Delta^{\!
c}$. Hence, by the convexity of $\Delta^{\! c}$, we have that
$F(\myx)$ points towards the interior of $\Delta^{\! c}$. As a result,
$\Delta^{\! c}$ is invariant by \eqref{eqn:THE_field}.
\end{proof}

\begin{proof}[Proof of Theorem~\ref{theorem2}.$(\mathit{ii})$]
  Introducing the new variables  $x_{uv}^i=\alpha
x_{u}^i- \alpha x_{v}^i$, $i\in C,\, u,v \in V$, $u\neq v$, in
\eqref{eqn:THE_field} leads to the system of $c \cdot \edges$
differential equations
\begin{equation}\label{eq-2a} 
\dot{x}_{uv}^i = -x_{uv}^i +\dfrac{\alpha }{\edges}\sum_{w\in
  V{\setminus}\{u\}}  \varphi\left(  \sum_{{k\in C{\setminus}\{i\}}}
  x_{uw}^k  \right) -  \dfrac{\alpha }{\edges}\sum_{w\in
  V{\setminus}\{v\}}  \varphi\left(  \sum_{k\in C{\setminus}\{i\}}
  x_{vw}^k  \right).
\end{equation}
By using \ref{odd_Hyp}, we may rewrite \eqref{eq-2a} in the form 
\begin{equation}\label{eq-2b}
\begin{aligned}
  \dot{x}_{uv}^i
  =&
    -x_{uv}^i +\dfrac{\alpha}{\edges}\sum_{w\in
  V{\setminus}\{u\}}  \varphi\left(  \sum_{{k\in C{\setminus}\{i\}}}
  x_{uw}^k  \right) + \dfrac{\alpha}{\edges}\sum_{w\in
  V{\setminus}\{v\}}  \varphi\left(  \sum_{k\in C{\setminus}\{i\}}
    x_{wv}^k  \right) \\
    &- \dfrac{\alpha}{\edges}\sum_{w\in V{\setminus}\{v\}} 1.
\end{aligned}
\end{equation}
Setting $\myx_{uv} = \left[ x_{uv} ^1, x_{uv}^2,\ldots,
  x_{uv}^c\right]^{t}$, $u\ne v$, and rewriting
\eqref{eq-2b} in vectorial form  leads to
\begin{equation}\label{eq-vf}
\begin{aligned}
\dot{\myx}_{uv} 
=& - \myx_{uv} +
\dfrac{\alpha}{\edges}\sum_{w\in V{\setminus}\{u\}} {\Phi}(B
\myx_{uw})   + \dfrac{\alpha}{\edges}\sum_{w\in
  V{\setminus}\{v\}} {\Phi}(B \myx_{wv})                          \\
&- \dfrac{\alpha}{\edges}\sum_{w\in V{\setminus}\{v\}}\mathbf{1}, 
\qquad  u\neq v,
\end{aligned} 
\end{equation}
where $\mathbf{1} = [1,1,\ldots,1]^t$, $B = (b_{ij})$ is the
non-singular $c\times c$ matrix with entries $b_{ij}= 1 - \delta_{ij}$,
where $\delta_{ij} = 1$ if $i=j$ and $\delta_{ij} = 0$ otherwise, and
${\Phi}:\R^c\to\R^c$ is the map defined by $\Phi([p_1, p_2, \ldots,
p_c]^t) = [\varphi(p_1), \varphi(p_2), \ldots, \varphi(p_c)]^t$.
Introducing the new variables $\myy_{uv}= \left[ y_{uv} ^1,
y_{uv}^2,\ldots, y_{uv}^c\right]^{t} = B \myx_{uv}$, $u\neq v$, in
\eqref{eq-vf} yields
\begin{equation}\label{yvec}
 \dot{\myy}_{uv} = -\myy_{uv} +
 \dfrac{\alpha}{\edges}\sum_{w\in V{\setminus}\{u\}} B{\Phi}(  \myy_{uw})
 + 
\dfrac{\alpha}{\edges}\sum_{w\in V{\setminus}\{v\}}
B{\Phi}( \myy_{wv}) -\dfrac{\alpha}{\edges}
\sum_{w\in V{\setminus}\{v\}}B\mathbf{1}.  
\end{equation}
After rewriting \eqref{yvec} in the scalar form, we obtain 
\begin{equation}\label{ys}
\begin{aligned}
\dot{y}_{uv}^i 
  =& 
  - y_{uv}^i + \dfrac{\alpha}{\edges} \sum_{w\in
  V{\setminus}\{u\}}\sum_{k\in C{\setminus}\{i\}}  
  \varphi\big(y_{uw}^k\big) + \dfrac{\alpha}{\edges} 
  \sum_{w\in V{\setminus}\{v\}}\sum_{k\in
  C{\setminus}\{i\}} \varphi\big(y_{wv}^k\big)        \\
  &- \dfrac{\alpha}{\edges} (d - 1 )(c-1).
\end{aligned}
\end{equation}
Again, let us introduce the new variables 
\begin{equation}\label{eq-H}
z_{iuv} = y_{uv}^i, \quad i\in C,\quad u,v\in V,\quad u\neq v.
\end{equation}
Then we may rewrite \eqref{ys} in the form
\begin{equation}\label{eq-H2}
\dot{z}_{iuv} = -z_{iuv} + \sum_{(j,r,s)\in\Lambda} w_{iuv}^{jrs}\, \varphi(z_{jrs})  + m
\end{equation}
where
\begin{equation}\label{eq-W}
\begin{gathered}
\Lambda=\{(i,r,s): i\in C,\, r,s\in V, \,r\neq s\}, \\[.75em]
  m
  =
  - \dfrac{\alpha}{\edges} (d-1 )(c-1),\quad
  w_{iuv}^{jrs}
  = \begin{cases}
      \dfrac{2\alpha}{\edges}
      &\textrm{if} \quad j\neq i, \,r = u, \,s = v\\[0.1in]
      \dfrac{\alpha}{\edges}
      &\textrm{if} \quad j\neq i, \,r = u, \,s\neq v  \\[0.1in]
      \dfrac{\alpha}{\edges}
      & \textrm{if}\quad  j\neq i, \,r\neq u, \,s = v  \\[0.1in]
      \phantom{1}0
      &\textrm{if}\quad j\neq i, \,r\neq u, \,s\neq v\\[0.1in]
      \phantom{1}0 & \textrm{if}\quad j=i
    \end{cases}
\end{gathered}
\end{equation}

It is clear that the multidimensional matrix $W =
\big(w_{iuv}^{jrs}\big)$ is symmetric in the sense that $w_{iuv}^{jrs}
= w_{jrs}^{iuv}$ for all $(i,u,v)\in \Lambda$, $(j,r,s)\in \Lambda$.

The system of equations \eqref{eq-H2} is analogous to the one
governing the dynamics of a Hopfield network. Provided that
$W=\big(w_{iuv}^{jrs}\big)$ is symmetric, \cite[p. 2090, Equation
(7)]{WOS:A1984SU47600030} proved (see Appendix B) that the map $L:\mathbb{R}^{c\vert E\vert}
\to\R$ defined by
\begin{align*}
  L ( {z})
  =
  -\frac12\sum_{(i,u,v)\in\Lambda}
    &\sum_{(j,r,s)\in\Lambda} w_{iuv}^{jrs}\, 
    \varphi\big(z_{iuv}\big) \varphi\big(z_{jrs}\big)\\
   &+
    \sum_{(i,u,v)\in\Lambda} \int^{\varphi(z_{iuv})}
  \varphi^{-1}(\eta) \mathrm{d}\eta  +
 \sum_{(i,u,v)\in\Lambda}   m \varphi(z_{iuv}), 
\end{align*}
is strict Lyapunov function for \eqref{eq-H2}. In this way, if $z: t\in [0,\infty)\mapsto z(t)\in\mathbb{R}^{c\vert E\vert} $ is a non-constant solution of  \eqref{eq-H2}, then 
the function $t\in [0,\infty)\mapsto L\big(z(t)\big)$ is strictly decreasing. Now we will explain how to translate such a result about the solutions of \eqref{eq-H2} into a result about the solutions of \eqref{eqn:THE_field}. Let $\Psi: \mathbb{R}^{dc}\to\mathbb{R}^{c\vert E\vert}$ be the map that at $x=(x_u^i)\in\mathbb{R}^{dc}$ takes the value
$z=(z_{iuv})\in\mathbb{R}^{c\vert E\vert}$, where
\begin{equation}\label{Psi} 
z = \Psi(x),\quad z_{iuv} =  \sum_{k\in C{\setminus}\{i\}} \big(\alpha x_u^k -\alpha
  x_v^k\big),\quad i\in C, \quad u,v\in V, \quad u\neq v.
\end{equation}
It follows from all the change of variables made within this section that if $x: t\in [0,\infty)\mapsto x(t)\in\mathbb{R}^{dc}$ is a solution of \eqref{eq-H2} then
$z:[0,\infty)\to \mathbb{R}^{c\vert E\vert}$ defined by $z(t)=\Psi\big(x(t)\big), t\in\mathbb{R}$, is a solution of \eqref{eq-H2}. To show that $L\circ\Psi$ is a Lyapunov function for    \eqref{eqn:THE_field} on $\Delta^c$, we will show that if $x: t\in [0,\infty)\mapsto x(t)\in \Delta^c$ is a non-constant solution of  \eqref{eqn:THE_field} on $\Delta^c$, then
$t\in [0,\infty)\mapsto L\Big(\Psi\big(x(t)\big)\Big)$ is strictly decreasing. To proceed with that reasoning, let $x: t\in [0,\infty)\mapsto x(t)\in \Delta^c$ be a non-constant solution of  $\eqref{eqn:THE_field}$ and let $z:t\in [0,\infty)\mapsto z(t)=\Psi\big(x(t)\big)$ be the corresponding solution of \eqref{eq-H2}. The proof will be finished if we show that
$z$ is non-constant because in that case, as already remarked, the function $t\in [0,\infty)\mapsto L\big(z(t)\big) = L\Big(\Psi\big(x(t)\big)\Big) $ is strictly decreasing. By way of contradiction, assume that $z$ is constant. By \eqref{Psi} and \eqref{pi}, there exist $b_{u}^i\in\mathbb{R}$, $i\in C$, $u\in V$, such that
\begin{equation}\label{bui} 
b_{u}^i = \pi_u^i\big(x(t)\big),\quad   t\ge 0, \quad i\in C, \quad u\in V.
\end{equation}
In this way, $b:=(b_u^i) = \pi \big(x(t)\big)\in\Delta^c$. In particular, we have that
\begin{equation}\label{b40}
\sum_{u\in V} b_u^i =1, \quad \forall i\in C. 
\end{equation}
By  \eqref{eqn:THE_field} and \eqref{bui}, we obtain that $x:t\in [0,\infty)\mapsto x(t)\in\mathbb{R}^{dc}$ satisfies the uncoupled system of differential equations:
$$
  \dot{x}_u^i(t) =  -x_u^i(t) + b_u^i, \quad   t\ge 0, \quad   i\in C, \quad   u\in V.
$$  
Applying integrating factor to the previous ode, we conclude that there exist $a_u^i\in\mathbb{R}$, $i\in C$, $u\in V$ such that
\begin{equation}\label{xuixui}
x_u^i(t) = a_u^i e^{-t} + b_u^i, \quad t\ge 0, \quad i\in C,\quad u\in V.
\end{equation}
Now given,  $u,v\in V$ with $u\neq v$, let $\myx_{uv}(t) = \left[ x_{uv} ^1(t), x_{uv}^2(t),\ldots,
  x_{uv}^c(t)\right]^{t}$, where $x_{uv}^i(t) = \alpha x_u^i(t) - x_v^i(t)$. Since, by hypothesis, $z$ is constant, we have that $\dot{z}(t)=0$ for all $t\ge 0$. Hence, by \eqref{Psi}, we have that
$$
0 = \dot{z}_{iuv}(t)= \sum_{k\in C{\setminus}\{i\}} \big(\alpha \dot{x}_u^k(t) -\alpha
  \dot{x}_v^k(t)\big) = \left[B \dot{x}_{uv}(t)\right]_i, \quad t\ge 0, \quad  i\in C.
$$
That is, $B \dot{x}_{uv}(t)=0$ for all $t\ge 0$. Since $B$ is non-singular, we have that $\dot{x}_{uv}(t)=0$ for all $t\ge 0$ or, equivalently, 
\begin{equation}\label{0111}
\dot{x}_u^i(t) = \dot{x}_v^i(t), \quad   t\ge 0,\quad   i\in C, \quad   u,v\in V, \quad u\neq v.
\end{equation}
Derivating \eqref{xuixui} and using \eqref{0111} imply that that there exist $a^i$, $i\in C$, such that
$$
x_u^i(t) = a^i e^{-t} + b_u^i, \quad t\ge 0, \quad i\in C,\quad u\in V.
$$
Summing over $u\in V$ and using \eqref{b40} yield
$$ 1 = a^i e^{-t} + 1, \quad \forall t\ge 0,\quad \forall i\in C,
$$
which is a contradiction. In fact, since $x$ is a non-constant solution, we have that $a^i\neq 0$ for at least one $i\in C$.
\end{proof}

\begin{definition}[chain-recurrent set]\label{dcs} Let $\delta >0$ and
$T>0$. A $(\delta, T)$-pseudo orbit of \eqref{eqn:THE_field} from $x
\in \D$ to $y \in \D$ is a finite sequence of partial orbits
$\{\phi_t(y_i) : 0 \leq t \leq t_i\}$; $i=0, \ldots, k-1$; $t_i\geq T$
of \eqref{eqn:THE_field} such that
$$\|y_0 - x\| < \delta, \quad \|\phi_{t_i}(y_i) - y_{i+1}\| <
\delta, \,\, i=0, \ldots, k-1,\quad \textrm{and} \quad y_k = y.$$ A
point $x \in \D$ is chain-recurrent if for every $\delta>0$ and $T>0$
there is a $(\delta, T)$-pseudo orbit from $x$ to itself. The
chain-recurrent set of \eqref{eqn:THE_field},  {denoted by $\CR$,}  is
the set of all chain-recurrent points.
\end{definition}

\begin{corollary}\label{cor1} The chain-recurrent set $\CR$ of
\eqref{eqn:THE_field} equals the set of fixed points of the map $\pi$.
\end{corollary}
\begin{proof} The set of fixed points of $\pi$, $\Fix(\pi)$,
equals the set $\Eq$ of equilibrium points of
\eqref{eqn:THE_field}. It is an elementary fact that
$\Eq\subset\CR$. By hypothesis, $\Fix(\pi)$, and therefore 
$\Eq$, is finite. Hence, if $L$ is the strict Lyapunov function in
Theorem \ref{theorem2}, then $L(\Eq)$ is finite. In this case, by
Proposition 3.2 in \cite{B96}, it follows that $\CR\subset\Eq$. We have
proved that $\CR = \Eq = \textrm{Fix}\, (\pi)$.
\end{proof}

\begin{definition}[limit set]\label{limitsetdef} Let $X=\{X(n)\}_{n\ge
0}$ be the process given by \eqref{process2}, which is defined on the
probability space $(\Omega, \F,\Prb)$. Given $\omega\in\Omega$,
we say that $x\in \D$ is a limit point of $\{X(n,\omega)\}_{n\ge 0}$
if there exists a strictly increasing sequence of integers $1\le
n_1<n_2<\cdots$ such that $\lim_{k\to\infty} X(n_k,\omega)=x$.
\end{definition}

\begin{proof}[Proof of Theorem \ref{theorem1}] By the ODE method for
stochastic approximation (see  Appendix A), there exists a full
measure set $\Omega'\subset\Omega$ such that for each $\omega\in\Omega'$,
the limit set of $\left\{ X(n,\omega)\right\}_{n\ge 0}$ is a connected
subset of the set $\CR$ of chain-recurrent points of
\eqref{eqn:THE_field}. By Corollary \ref{cor1}, the connected
components of $\CR$ consists of the fixed points of $\pi$, which are
finitely many. Hence, the limit set of $\left\{
X(n,\omega)\right\}_{n\ge 0}$ is a set consisting of a fixed point of
$\pi$. Since the limit set is a unique point, we conclude that
$\lim_{n\to\infty} X(n,\omega)$ exists and belongs to
$\textrm{Fix}\,(\pi)$.
\end{proof}

\section{Proof of Corollary \ref{corollary1}}\label{sec:corol_proof}

To prove Corollary \ref{corollary1}, we will show that if $\vert \alpha\vert$ is enough small, then the map $\pi$ defined in \eqref{pi} is a contraction on $\D$ with respect to the norm $\Vert x\Vert_1 =  \sum_{(u,i)\in V\times C} \vert x_u^i \vert$, that is,
$\pi$ satisfies the following conditions:
\begin{itemize}
\item [(i)] $\pi$ takes the compact, convex set $\D$ into itself;
\item [(ii)] there exists $0<\lambda<1$\,\, such that \,\, $\Vert \pi(x) - \pi(y)\Vert_1\le\lambda \Vert x-y\Vert_1$, \quad $\forall x,y\in\D$.
\end{itemize} 
Condition (i) was already verified, see the proof of Theorem \ref{theorem2}. To show that Condition (ii) is true, let $J\pi(x) = \big((\partial \pi_u^i/\partial x_v^j)(x)\big)$ denote the Jacobian matrix of $\pi$ at $x\in \D$. Since $\D$ is compact and convex, by the Mean Value Theorem, $\pi$ will be a contraction on $\D$ provided
$\sup_{x\in\D} \left\Vert J\pi(x)\right\Vert_1<1$ or equivalently, if the following condition is true
\begin{itemize}
\item[(iii)]
$\displaystyle \sum_{(u,i)\in V\times C} \left\vert \dfrac{\partial \pi_u^i}{\partial x_v^j}(x) \right\vert <1$, $\quad\forall (v,j) \in V\times C$, $\forall x \in \D$.
\end{itemize}
In fact, for all  $(v,j) \in V\times C$ and $x\in\D$, we have that
$$ \dfrac{\partial \pi_u^i}{\partial x_v^j}(x) = \begin{cases} \phantom{-} 0 & \textrm{if} \quad i=j\\
\phantom{-}\sum_{w\in V{\setminus} \{u\}} \dfrac{\alpha}{\vert E\vert} \varphi'\left(  \sum_{k\in C{\setminus}\{i\}}  \left(\alpha x_{u}^k -
    \alpha x_{w}^k\right)    \right)  & \textrm{if} \quad i\neq j,\,\, u=v\\-\dfrac{\alpha}{\vert E\vert}
  \varphi'\left(  \sum_{k\in C{\setminus}\{i\}}  \left(\alpha x_{u}^k -
    \alpha x_{v}^k\right)    \right) & \textrm{if} \quad i\neq j,\,\, u\neq v\
\end{cases}.
$$ 
In this way, for all  $(v,j) \in V\times C$ and $x\in\D$, since $\edges= d(d-1)/2$, we have that
\begin{align*}
\sum_{(u,i)\in V\times C}
\left\vert \dfrac{\partial \pi_u^i}{\partial x_v^j}(x) \right\vert
&\leq
2(c-1) (d-1) \dfrac{\vert\alpha\vert}{\vert E\vert} \sup_{t\in \R} \left\vert\varphi'(t)\right\vert \\
&=
4(c-1)  \dfrac{\vert\alpha\vert}{d} \sup_{t\in \R} \left\vert\varphi'(t)\right\vert<1.
\end{align*}
Hence, Condition (iii) holds and $\pi$ is a contraction on $\D$. Now since $\D$ is a compact (therefore complete) subset of $\mathbb{R}^{dc}$, 
by the Banach Fixed Point Theorem, we have that $\pi$ has a unique fixed point. We claim that such a fixed point is $x^*=\left(\frac1d,\ldots,\frac1d\right)$. In fact, by $(h2)$, we have that $\varphi(0)=\frac12$. By \eqref{pi}, it follows that 
\[
\pi_u^i(x^*) = 
\dfrac{1}{\vert E\vert } \sum_{w\in V{\setminus}\{u\}}
\varphi\left(  \sum_{k\in C{\setminus}\{i\}}  \left( \dfrac{\alpha}{d} -
     \dfrac{\alpha}{d}\right)    \right) = \dfrac{2}{d(d-1)} (d-1)\varphi(0) = \dfrac1d = (x^*)_u^i.
\]
To conclude the proof, apply Theorem \ref{theorem1}.

\section{Proof of Example \ref{example1}}\label{sec:exampl_proof}

First we provide conditions under which a point $x=(x_1^1,
x_2^1,x_3^1,x_1^2, x_2^2, x_3^2)\in\Delta^2$ is a fixed-point of
$\pi$. We write $x=(y,z)$, where
\[
  y = (y_1,y_2,y_3) = (x_1^1, x_2^1, x_3^1)
  \quad\textrm{and}\quad z =
  (z_1,z_2,z_3) = (x_1^2, x_2^2, x_3^2). 
\]
By \eqref{pi}, we have that $x=(y,z)$ is a fixed-point of $\pi$ if and
only if the following system of equations is satisfied
\begin{equation}\label{dupla}
    \begin{cases}
      y_u = \dfrac{1}{3}\sum_{w\in \{1,2,3\}{\setminus} \{u\}}
      \varphi\big(\alpha z_u -\alpha z_w\big)\\[0.1in] 
      z_u = \dfrac{1}{3}\sum_{w\in \{1,2,3\}{\setminus} \{u\}}
      \varphi\big(\alpha y_u -\alpha y_w\big) 
    \end{cases},\quad u\in \{1,2,3\}. 
\end{equation}
\begin{lemma}\label{lema67} Let $x=(y,z)\in\Delta^2$ be such that $z=y$ and
\begin{equation}\label{yu}
y_u = \dfrac{1}{3}\sum_{w\in \{1,2,3\}{\setminus} \{u\}} \varphi\big(\alpha z_u -\alpha z_w\big),\quad  u\in \{1,2,3\}.
\end{equation}
Then $x$ is a fixed-point of $\pi$.
\end{lemma}
\begin{proof} By hypothesis the first three equations in \eqref{dupla}
are satisfied. To prove that the fourth one is true, let $u=1$. Then,
we have that
\begin{eqnarray*} 
z_1 = y_1 &=& \dfrac13 \varphi\big(\alpha z_1 - \alpha z_2 \big) +  \dfrac13 \varphi\big(\alpha z_1 - \alpha z_3 \big)=\dfrac13 \sum_{w\in\{1,2,3\}{\setminus}\{1\}}\varphi\big(\alpha y_1 - \alpha y_w \big),
\end{eqnarray*}
which shows that the fourth equation in \eqref{dupla}
holds. Proceeding in the same way for $u=2$ and $u=3$, we obtain that
the fifth and the sixth equations in \eqref{dupla} are true, implying
that $x=(y,z)=(y,y)$ is a fixed-point of $\pi$.
\end{proof}

In what follows, given a permutation $\sigma$ on the set $V=\{1, 2,
3\}$ and a vector \linebreak $w=(w_1,w_2,w_3)$, we use the notation
$\omega^\sigma=\big(\omega_{\sigma(1)}, \omega_{\sigma(2)},
\omega_{\sigma(3)}\big)$.

\begin{lemma}\label{lema68}  If $x=(y,z)\in\Delta^2$ satisfies  \eqref{yu}, then for any permutation $\sigma$ on the set $V= \{1, 2, 3\}$, we have that
$x^* = (y^\sigma, z^\sigma)$ also satisfies \eqref{yu}.
\end{lemma}
\begin{proof} Suppose that $x=(y,z)$ satisfies  $\eqref{yu}$ and let $\sigma$ be a permutation on the set $V=\{1,2,3\}$. Let $x^*=(y^*, z^*)$, where $y^*=y^\sigma$ e $z^*=z^\sigma$. Then, we have that
\begin{eqnarray*}
y_u^* = y_{\sigma(u)} &=&  \dfrac{1}{3}\sum_{w\in \{1,2,3\}{\setminus} \{\sigma(u)\}} \varphi\big(\alpha z_{\sigma(u)} -\alpha z_w\big)\\
  &=&  \dfrac{1}{3}\sum_{w\in \{1,2,3\}{\setminus} \{u\}} \varphi\big(\alpha z_{\sigma(u)} -\alpha z_{\sigma(w)}\big)\\
  &=&  \dfrac{1}{3}\sum_{w\in \{1,2,3\}{\setminus} \{u\}} \varphi\big(\alpha z_u^* -\alpha z_w^* \big)
\end{eqnarray*}
In this way,   $x^*=(y^*, z^*)=(y^\sigma, z^\sigma)$ satisfies \eqref{yu}.
\end{proof}

\begin{proof}[Proof of Example \ref{example1}] First we consider the  point 
\begin{equation*}\label{pointx}
 x = (y,z) = \left(\dfrac{23}{615}, \dfrac13, \dfrac{129}{205}, \dfrac{23}{615}, \dfrac13, \dfrac{129}{205}\right).
\end{equation*}
To show that $x$ is a fixed-point of $\pi$, we will verify the hypotheses of Lemma \ref{lema67}. First notice that $z=y$. Moreover,  it is elementary to show that if $\alpha = \frac{615}{182}\ln\,9$, then the system of equations in \eqref{lema67} is verified. Hence, by \mbox{Lemma \ref{lema67},} $x$ is a fixed-point of $\pi$. Now let $x^*$ be any of the other $5$ points in \eqref{eqn:example_fp}. As $x^* = (y^*, z^*) = (y^\sigma, z^\sigma)$, from Lemmas \ref{lema67} and \ref{lema68} it follows that $x^*$ is a fixed-point of $\pi$. 

Now it remains to prove that at each of the $6$ fixed-points, each eigenvalue of $J\pi(x)$ has absolute value less than $1$. By \eqref{pi},  we have that at each any point $x=(y,z)$ of $\Delta^2$,
$$ \pi_u^i(x) =  \pi_u^i(y,z) =
    \begin{cases}
    \dfrac{1}{3}\sum_{w\in \{1,2,3\}{\setminus} \{u\}} \varphi\big(\alpha z_u -\alpha z_w\big) &\textrm{if} \quad i=1
    \\[0.1in]
 \dfrac{1}{3}\sum_{w\in \{1,2,3\}{\setminus} \{u\}} \varphi\big(\alpha y_u -\alpha y_w\big)&\textrm{if} \quad i=2
    \end{cases},\quad u\in \{1,2,3\}.
$$
In particular, the Jacobian matrix of $\pi$ at $x=(y,y) $ is given by
\begin{equation}\label{forma}
 J\pi(x) = J\pi(y,y) = \begin{pmatrix} 0 & M(y)\\
M(y) & 0
\end{pmatrix},
\end{equation}
where $0$ denotes the null matrix of order $3$, 
\[
 M(y)
 =
 \begin{pmatrix}
    \frac{\alpha}{3} \varphi'\big(y_{12}\big) +
    \frac{\alpha}{3} \varphi'\big(y_{13}\big)  &
   -\frac{\alpha}{3} \varphi'(y_{12}) &
   -\frac{\alpha}{3}\varphi'(y_{13}) \\[.5em] 
   -\frac{\alpha}{3} \varphi'(y_{21}) &
    \frac{\alpha}{3}\varphi'\big(y_{21}\big) +
    \frac{\alpha}{3}\varphi'\big(y_{23}\big)&
   -\frac{\alpha}{3} \varphi'(y_{23})  \\[.5em]
   -\frac{\alpha}{3} \varphi'(y_{31}) &
   -\frac{\alpha}{3} \varphi'(y_{32})  &
    \frac{\alpha}{3} \varphi'(y_{31}) +
    \frac{\alpha}{3} \varphi'(y_{32}) 
\end{pmatrix},
\]
and $y_{uw} = \alpha y_u - \alpha y_w$ for all $u,w\in \{1,2,3\}$ with $u\neq w$.  

Since $\varphi(t)=1/(1+e^{-t})$, by $(h2)$, $\varphi'(t) =
(1-\varphi(t))\varphi(t)=\varphi'(-t)$. Therefore, since $y_{uw}=-y_{wu}$,
we have that $M(y)$ is a symmetric matrix.  Now we will do the
calculation for the point
\begin{equation*}
 x = (y,y) =  \left(\dfrac{23}{615}, \dfrac13, \dfrac{129}{205},
   \dfrac{23}{615}, \dfrac13, \dfrac{129}{205}\right).
\end{equation*}
Since $\alpha = \frac{615}{182}\ln 9$, we have that
\[
  y_{12}=-\ln 9,  \quad y_{13}=-2\ln 9, \quad y_{23} = -\ln 9.
\]
Therefore, since $\varphi'(t) = (1-\varphi(t))\varphi(t)$, we have
that
\begin{equation}\label{MMy}
  M(y)
  = \frac{\alpha }{3}
  \begin{pmatrix}
    \phantom{-}\frac{8577}{84050} & -\frac{9 }{100} & -\frac{81 }{6724} \\[.5em]
    -\frac{9 }{100} &  \phantom{-}\frac{ 9}{50} & -\frac{9 }{100} \\[.5em]
    -\frac{81 }{6724}  & -\frac{9 }{100}   & \phantom{-}\frac{8577}{84050}
  \end{pmatrix}
\end{equation}
We will need the following lemma.

\begin{lemma}\label{square} If $\lambda$ is a real eigenvalue of $J\pi(x)$, then $\lambda^2$ is a real eigenvalue of $\big(M(y)\big)^2$.
\end{lemma}
\begin{proof} Let $\lambda$ be an eigenvalue of $J\pi(x)$ associated
to an eigenvector
$$v=\begin{pmatrix} p\\ q\end{pmatrix},\quad
\textrm{where}\quad p=\begin{pmatrix} p_1\\ p_2\\
p_3\end{pmatrix}\quad\textrm{and}\quad q=\begin{pmatrix} q_1\\ q_2\\
                                           q_3\end{pmatrix}.$$
Then, by \eqref{forma}, we have that
$$ \begin{pmatrix} \lambda p\\ \lambda q \end{pmatrix} = \lambda v = J\pi(x)v = \begin{pmatrix} 0 & M(y)\\
M(y) & 0
\end{pmatrix}\begin{pmatrix} p\\q\end{pmatrix}=\begin{pmatrix} M(y) q\\ M(y)p\end{pmatrix}.$$ 
In particular, we have that
$$ \big(M(y)\big)^2 p = M(y) M(y) p = M(y) \lambda q = \lambda M(y) q = \lambda^2 p,
$$
showing that $\lambda^2$ is eigenvalue of $\big(M(y)\big)^2$.
\end{proof}
By \eqref{MMy}, the eigenvalues of $M(y)$ are approximately $0$,
$0.28$ and $0.66$. Therefore, the eigenvalues of $\big(M(y)\big)^2$
are real and less than $1$. By \eqref{forma}, $J\pi(x)$ is symmetric,
therefore the eigenvalues of $J\pi(x)$ are real. By Lemma
\ref{square}, the eigenvalues of $J\pi(x)$ are real numbers with
absolute value less than $1$. More precisely, if $\lambda$ is
eigenvalue of $J\pi(x)$, then $\lambda$ is real and $\lambda^2\in
[0,1)$. This implies that $\lambda\in (-1,1)$. Hence, the spectral
radius of $J\pi(x)$ is less than $1$.  The same procedure can be
applied to the other $5$ fixed-points of $\pi$.
\end{proof}

\appendix

\section{Stochastic approximations}

In what follows, denote by $(\Omega,\F, \mathbb{P})$ the
probability space where the process $X=\{X(n)\}_{n\ge 0}$ given by
\eqref{process2} is defined. Set $\F_0=\{\emptyset,
\Omega\}$ and for each $n\ge 1$, let $\F_n$ be the
$\sigma$-algebra generated by the process  {$X$ from time 1} up to
time $n$. We denote by $F:\mathbb{R}^{dc}\to\mathbb{R}^{dc}$ the
vector field defined in \eqref{eqn:THE_field}.
 
\begin{lemma}\label{plem} The process $X=\{X(n)\}_{n\ge 0}$ defined by
\eqref{process2} satisfies the recursion
\begin{equation}\label{eqn:SA}
  X(n+1) - X(n) =  {\gamma_{n}}\big(F(X(n)) +  {U(n)}\big),\quad
  n\ge  0,
\end{equation}
where
\begin{enumerate}[nosep]
\item [$(i)$] $\{\gamma_n\}_{n\ge 0}$ is a deterministic sequence of
  positive numbers;  
\item [$(ii)$] $\sum_{n\ge 0} \gamma_n=\infty$ \,\,\textrm{and}\,\,
$\sum_{n\ge 0} \gamma_n^2<\infty$;
\item [$(iii)$] $U(n):\Omega\to \mathbb{R}^{dc}$ is a
  $\F_n$-measurable map for each $n\ge 0$;
\item [$(iv)$] $\E[ {U(n)} \mid \F_n]=0$ for all  $n\ge 0$;
\item [$(v)$] The sequence $\left\{\sum_{k=0}^n \gamma_k
    U(k)\right\}_{n\ge 0}$ converges (a.s.).
\end{enumerate} 
\end{lemma}

\begin{proof}[Proof of $(i)$-$(iv)$]
Using \eqref{rule3} define
\begin{equation}\label{eqn:Niv}
  N^i_u(n+1) = \sum_{w \in V{\setminus}\{u\}} \mathbf{1}_{B^i_{w\to 
      u}(n+1)}.
\end{equation}
That is, $N^i_u(n+1)$ is the number of balls of colour $i$ put into
the urn at vertex $u$ at time $n+1$. We recall that the number of
balls of colour $i$ at urn $u$ at time $n$ is denoted by
{$B_u^i(n)$}. Moreover, as an hypothesis of the model, we assumed that
there exists $b_0\ge 1$ such that $\sum_{u\in V} B_u^i(0)=b_0$ for all
$i\in C$. In this way, for each $i \in C$ and $n\ge 0$,
\[
  \sum_{u \in V} B_u^i(n) = b_0 + n\edges.
\]
By \eqref{process}, the increment in the proportion of balls of colour
$i$ at vertex $u$ at time $n+1$ equals
\begin{equation}\label{eqn:increment}
 \begin{aligned}
  X^i_u(n+1)-X^i_u(n)
  &=
    \frac{B^i_u(n) + N^i_u(n+1)}{b_0 + (n+1)\edges} -
    \frac{B^i_u(n)}{b_0+n\edges}                            
  =
    \frac{-\edges X^i_u(n) + N^i_u(n+1)}{b_0 + (n+1)\edges} \\ 
  &=
    \frac{1}{b_0/\edges +n+1}\Big(-X^i_u(n) + \frac{1}{\edges}
    N^i_u(n+1)\Big).
  \end{aligned}
\end{equation}
Set 
\begin{equation}\label{eqn:xi}
 {\gamma_{n}} = 1/(b_0/\edges + n+1) \quad\textrm{and}\quad
\xi^i_u( {n}) = \frac{1}{\edges} N^i_{u}(n+1).
\end{equation}
For each $ {n\ge 1}$, define $\xi(n) = \big(\xi^1_1(n)$, $\ldots$,
$\xi^1_d(n)$, $\ldots$ $\xi^c_1(n)$, $\ldots$, $\xi^c_d(n)\big)$. By
considering the expression for the increments of each component of $X$
given by \eqref{eqn:increment}, it follows that  {for any $n\ge 0$,}
\begin{equation}\label{XX}
  X(n+1) - X(n)
  =
  \gamma_{ {n}} \big(-X(n) + \xi( {n})\big).
\end{equation}
According to \eqref{eqn:xi} and \eqref{eqn:Niv}, for any $i
\in C$,
\begin{align}\label{EEE}
\begin{split}
  \E\big[\xi^i_u( {n}) \mid \F_n\big] 
  = 
  \frac{1}{\edges} \sum_{w \in V{\setminus}\{u\}}
  \E\big[\mathbf{1}_{B^i_{w\to u}(n+1)}\mid \F_n\big]
  =
  \frac{1}{\edges} \sum_{w \in V{\setminus}\{u\}}\Prb\big(B^i_{w\to 
    u}(n+1) \mid \F_n\big).
\end{split} 
\end{align}
Markov's property implies that $\Prb\big(B^i_{w\to u}(n+1) \mid
\F_n\big) = \Prb\big(B^i_{w\to u}(n+1) \mid
 {X(n)}\big)$. Hence, by \eqref{EEE}, \eqref{rule3}, \eqref{pi},
we obtain
\begin{equation}\label{E4}
  \E\big[\xi^i_u( {n}) \mid \F_n\big] =  \frac{1}{\edges} \sum_{w
    \in  V{\setminus}\{u\}}\Prb\big(B^i_{w\to
    u}(n+1) \mid  {X(n)} \big) = \pi_u^i\big(X(n)\big).
 \end{equation}
By \eqref{E4} and \eqref{eqn:THE_field}, 
 \begin{equation}\label{-X2}
 -X(n) = F\big(X(n)\big) - \pi \big(X(n)\big)= F\big(X(n)\big) -
 \E\big[\xi( {n+1}) \mid \F_n\big].
 \end{equation}
Replacing $-X(n)$ in the right hand side of  {\eqref{XX}} by
\eqref{-X2} yields
  $$X(n+1) - X(n)
  =
  \gamma_{ {n}} \left(F\big(X(n)\big)+ U( {n})\right), 
  $$
where  {$U(n) = \xi(n) - \E[\xi(n)\mid \F_n]$}.
  
We have proved that the increments of $X$ satisfy
\eqref{eqn:SA}. Moreover, it follows from the expressions of
$\gamma_{n}$ and $U(n)$ that $(i)$-$(iv)$ hold.
\\

\noindent \textit{Proof of $(v)$:} Let $M_n = \sum_{k=0}^n \gamma_k 
U(k)$. An immediate verification shows that $\{M_n; n\geq 0\}$ is a
martingale adapted to the filtration $\{\F_{n+1}: n\geq 0\}$, that is
$\E[M_{n+1}\mid\F_{n+1}] = M_n$.  It then follows that  almost
surely
\begin{align*}
  {\sum_{k=0}^n} \E\big[\|M_{k+1} - M_k\|^2\big|\  {\F_{k+1}}\big]
  &=
     {\sum_{k=0}^n} \gamma_{k+1}^2  \E\big[\| {U(k+1)}\|^2\big|\  
     {\F_{k+1}}\big] \\
  &\leq
     {\sum_{k=0}^n} \gamma_{k+1}^2 \bigg(\sum_{u \in V}\sum_{i \in C}
     {\xi^i_u(k+1)}\bigg)^2  
  \leq
    \bigg(c d
    \frac{(d-1)}{\edges}\bigg)^2  {\sum_{k=0}^n} \gamma^2_{k+1} \\
  &\leq  
  4 c^2 \sum_{k\geq 0} \gamma^2_{k+1} < \infty.
\end{align*}
 {This implies that $M_n$ converges almost surely to a finite limit,
see Theorem 5.4.9 in \cite{D10}}, and hence that $\{M_n: n\geq 0\}$ is
a Cauchy sequence.  {This concludes the proof of
assertion $(v)$.}
\end{proof}

In what follows, we will need Definitions \ref{dcs} and
\ref{limitsetdef}. 

\begin{theorem} Let $X=\{X(n)\}_{n\ge 0}$ be the process given by
\eqref{process2}, which is defined on the probability space
$(\Omega, \F, \Prb)$. There exists a full measure set
$\Omega'\subset\Omega$ such that for each $\omega\in \Omega'$, the
limit set of $\{X(n,\omega)\}_{n\ge 0}$ is a connected set contained
in the chain-recurrent set of \eqref{eqn:THE_field}.
\end{theorem}
\begin{proof} This is a well established result for any process
satisfying the list of conditions in the statement of Lemma
\ref{plem}, see for instance \cite{B96}.
\end{proof}

\section{Hopfield's energy function} 

\begin{theorem}[Hopfield]\label{thmH} Let $\Lambda$ be a finite set
and $G=(G_\mu)_{\mu\in\Lambda}: \mathbb{R}^{\vert\Lambda\vert} \to
\mathbb{R}^{\vert\Lambda\vert}$ be the vector field whose
coordinate-function at $z=(z_{\mu})$ takes the value
\begin{equation}\label{G}
 G_{\mu}(z) = - z_{\mu} +\sum_{\nu\in\Lambda}  w_{\mu\nu}
 \varphi(z_{\nu}) + I_{\mu}, \quad \mu\in\Lambda,
\end{equation}
where $W=(w_{\mu\nu})$ is a symmetric real matrix, $I_{\mu}, \mu\in
\Lambda$, are constants and $\varphi:\mathbb{R}\to (0,1)$ is any
function satisfying the hypothesis \ref{deriv_Hyp}. Then the function
$L:\mathbb{R}^{\vert\Lambda\vert}\to\mathbb{R}$ defined by
\[
L (z) = -\frac12\sum_{\mu, \nu\in\Lambda}
  w_{\mu\nu}\, 
\varphi (z_{\mu} ) \varphi (z_{\nu} ) + \sum_{\mu\in\Lambda}
\int^{\varphi(z_{\mu})} \varphi^{-1}(\eta) 
\mathrm{d}\eta  +
\sum_{\mu\in\Lambda}   I_{\mu} \varphi(z_{\mu}), 
\]
 is a strict Lyapunov function for the ordinary differential equation 
\begin{equation}\label{ode3}
\dot{z} =  G(z)
\end{equation}
on $\mathbb{R}^{\vert\Lambda\vert}$ and it is called  Hopfield's
energy function. 
\end{theorem}
\begin{proof} The proof we present is adapted from
\cite{WOS:A1984SU47600030}, where it appears in a different
context and with another notation. By direct derivation, since
$W=(w_{\mu\nu})$ is symmetric, we have that for all $t\ge 0$,
\begin{equation}\label{L4}
\begin{split}
\dfrac{d}{dt} L\big(z(t)\big) 
= 
\sum_{\mu\in\Lambda}
\Bigg(-\sum_{\nu\in\Lambda} w_{\mu\nu}
  \varphi(z_{\nu}(t))&\varphi'(z_{\mu}(t))+
  z_{\mu}(t)\varphi'(z_{\mu}(t))                        \\
  &+ I_{\mu} \varphi'(z_{\mu}(t))\Bigg) 
  G_{\mu}\big(z(t)\big),
\end{split}
\end{equation} 
where $z(t)$ denotes the solution of \eqref{ode3} that satisfies
$z(0)=z^{(0)}\in\mathbb{R}^{\vert\Lambda\vert}$. By using \eqref{G},
we may rewrite \eqref{L4} in the form
\[
  \dfrac{d}{dt} L\big(z(t)\big) =
  -\sum_{\mu\in\Lambda}\varphi'\big(z_\mu(t)\big)
  \left[G_{\mu}\big(z(t)\big)  \right ]^2.
\]
If $z^{(0)}$ is not an equilibrium point of \eqref{ode3}, then
$\left[G_{\mu}\big(z(t)\big) \right ]^2>0$ for all $t\ge 0$. Moreover,
by the hypothesis \ref{deriv_Hyp}, we have
$\varphi'\big(z_\mu(t)\big)>0$ for all $t\ge 0$. This shows that $L$
is a strict Lyapunov function for \eqref{ode3} on
$\mathbb{R}^{\vert\lambda\vert}$. The other condition in hypothesis
\ref{deriv_Hyp} regarding the boundedness of $\varphi'$ is important
to ensure that the solutions of \eqref{ode3} are defined for all $t\ge
0$.
\end{proof} 

To use Theorem \ref{thmH} in our article, consider the notation used
in \eqref{eq-H2} and set
\begin{gather*}
  \Lambda = \{(i,u,v): i\in C, u,v\in V, u\neq v\},\\
  \mu=(i,u,v)\in\Lambda,\quad \nu=(j,r,s)\in\Lambda,\quad
  I_{\mu}=m, \quad w_{\mu\nu} = w_{iuv}^{jrs}.
\end{gather*}

\ \\

\noindent \textbf{Acknowledgements}. We acknowledge Paulo Jos\'e
  Rodriguez for writing the code of the simulations shown in
  Figure~\ref{fig:triangles}.

\bibliographystyle{amsplain}

\providecommand{\bysame}{\leavevmode\hbox to3em{\hrulefill}\thinspace}
\providecommand{\MR}{\relax\ifhmode\unskip\space\fi MR }
\providecommand{\MRhref}[2]{%
  \href{http://www.ams.org/mathscinet-getitem?mr=#1}{#2}
}
\providecommand{\href}[2]{#2}

\end{document}